\documentclass[12pt]{amsart}

\usepackage{graphicx}
\usepackage[colorlinks,citecolor=red,pagebackref,hypertexnames=false]{hyperref}
\usepackage{amsfonts}
\usepackage{amsmath,amssymb,amsthm,amssymb,amscd,cancel,stackrel,fullpage}
\usepackage{blkarray,multirow, mathtools,txfonts}
\usepackage{memhfixc}
\usepackage{latexsym}
\usepackage{tikz}
\usepackage{enumerate}
\usepackage{esint}
\newcommand{\mfa}{\mathfrak{a}}

\newcommand{\NN}{\mathbb{N}}
\newcommand{\ZZ}{\mathbb{Z}}

\DeclareMathOperator{\Inc}{Inc (\NN)}
\DeclareMathOperator{\Sym}{Sym}

\newcommand{\la}{\langle}
\newcommand{\ra}{\rangle}

\newtheorem{theorem}{Theorem}[section]
\newtheorem{lemma}[theorem]{Lemma}
\newtheorem{proposition}[theorem]{Proposition}
\newtheorem{corollary}[theorem]{Corollary}

\newtheorem{lem-def}[theorem]{Lemma and Definition}
\newtheorem{prop-def}[theorem]{Proposition and Definition}

\theoremstyle{definition}
\newtheorem{remark}[theorem]{Remark}
\newtheorem{rem-def}[theorem]{Remark and Definition}
\newtheorem{example}[theorem]{Example}

\newtheorem{question}[theorem]{Question}

\newcount\HOUR
\newcount\MINUTE
\newcount\HOURSINMINUTES
\newcount\INTVAL
\newcommand{\twodigit}[1]{\INTVAL=#1\relax\ifnum\INTVAL<10 0\fi\the\INTVAL}
\HOUR=\time\divide\HOUR by 60\relax
\HOURSINMINUTES=\HOUR\multiply\HOURSINMINUTES by 60\relax
\MINUTE=\time\advance\MINUTE by -\HOURSINMINUTES\relax
\newcommand\rightnow{
            \twodigit{\the\HOUR}:\twodigit{\the\MINUTE},
            \twodigit{\number\day}.\space
            \ifcase\month\or January\or February\or March\or April\or
May\or June\or July\or August\or September\or October\or November\or
December\fi
            \space\number\year}

\begin{document}

\author{Sema G\"unt\"urk\"un, Uwe Nagel}
\title{Equivariant Hilbert Series of Monomial Orbits} %invent your own title..

%\author{Sema G\"unt\"urk\"un}
\address{{\bf Sema G\"unt\"urk\"un, }Department of Mathematics, University of Michigan, 
530 Church Street East Hall, Ann Arbor, MI 48109, USA}
\email{gunturku@umich.edu}
%\author{Uwe Nagel}
\address{{\bf Uwe Nagel, }Department of Mathematics,
University of Kentucky,
715 Patterson Office Tower, Lexington, KY 40506-0027, 
 USA}
\email{uwe.nagel@uky.edu}

\thanks{The first author was partially supported by Simons Foundation grant \#317096.}


%\date{\rightnow} %..and date

\begin{abstract} 
The equivariant Hilbert series of an ideal generated by an orbit of a monomial under the action of the monoid $\Inc$ of strictly increasing functions is determined. This is used to find the dimension and degree of such an ideal. The result also suggests that the description of the denominator of an equivariant Hilbert series of an arbitrary $\Inc$-invariant ideal as given by Nagel and R\"omer is rather efficient. 
\end{abstract}

\maketitle %this command creates title for your paper. Try commenting it out to see what it does

%%%%%%%%%%%%%%

\section{Introduction}

For a polynomial ring over a field $K$ in finitely many variables, Hilbert showed that its ideals are finitely generated and the vector space dimensions of graded components  of its homogeneous ideals eventually grow polynomially. Equivalently, their Hilbert series are rational. Recently, analogs of these results have been established for certain ideals in polynomials rings in infinitely many variables. 

To describe this more precisely, fix a positive integer $c \ge 1$ and consider a polynomial ring $K[X] = K[X_{[c] \times \NN}] = K[x_{i,j} \; | \; 1 \le i \le c,\ 1 \le j]$. Let $\Inc$ be the monoid of strictly increasing functions on the set $\NN$ of positive integers
\[
\Inc = \{\pi: \NN \to \NN \; | \; \pi (i) < \pi (i+1) \text{ for all } i \ge 1\}. 
\]
Setting $\pi \cdot x_{j, k} = x_{j, \pi (k)}$ induces an action of $\Inc$ on $K[X]$. In \cite{C} and \cite{HS} it is shown that any $\Inc$-invariant ideal $I$ of $K[X]$ is generated by finitely many orbits.  This and related results are of great interest, for example, in algebraic statistics, the study of tensors, or in representation theory (see, e.g., \cite{Draisma-factor, DK,  HS, SS-14, S}). 

If $I$ is a homogeneous ideal,  in \cite{NR} an equivariant Hilbert series of $K[X]/I$ has been defined as a formal power series in two variables
\[
H_{K[X]/I} (s, t) =  \sum_{n \ge 0,\, j \ge 0} \dim_K [K[X_n]/I_n]_j \cdot s^n t^j, 
\]
where $K[X_n] = K[X_{[c] \times [n]}] = K[x_{i,j} \; | \; 1 \le i \le c,\ 1 \le j \le n]$ and $I_n = I \cap K[X_n]$. Note that any ideal of $K[X]$ that is invariant under the action of $\Sym (\infty)$ (induced by moving the column indices of the variables) is also $\Inc$-invariant. For any homogeneous $\Inc$-invariant ideal of $K[X]$,  it has been shown in \cite{NR} that its equivariant Hilbert series is rational of the form
\begin{equation}
    \label{eq: general equiv Hilb series}
H_{K[X]/I} (s, t) =  \frac{g(s, t)}{  (1-t)^a \cdot \prod_{j =1}^b \big[(1-t)^{c_j} - s \cdot f_j (t) \big]},
\end{equation}
where $a, b, c_j$ are non-negative integers with $c_j \le c$, \ $g (s, t) \in \ZZ[s, t]$, and each $f_j (t)$ is a polynomial in $\ZZ[t]$ satisfying $f_j (1) >  0$. 

This form has been used in \cite{NR} to show in particular that the dimension of $K[X_n]/I_n$  eventually grows linearly in $n$ and that the limit $\lim\limits_{n \to \infty} \sqrt[n]{\deg I_n}$ exists and is a positive integer. However, the equivariant Hilbert series is explicitly known for only a few ideals. Furthermore, a different argument for the rationality of the Hilbert series $H_{K[X]/I} (s, t)$ has been given more recently in \cite{KLS}, but without  a more precise description of the rational function. The authors wonder about a good description of its denominator. In order to begin addressing  these issues we consider any ideal $I$ that is generated by the $\Inc$ orbit of some monomial of $K[X]$. For ease of notation, let us focus on the case $c=1$ in this introduction and write $x_j$ for $x_{1, j}$. Let  $I$ be the ideal generated by the orbit of a monomial $x_{\mu_1}^{a_1}x_{\mu_2}^{a_2}\cdots x_{\mu_r}^{a_r}$, where $\mu_1 < ... < \mu_r$ and $r, a_1,..,a_r \in \NN$, which we write as $I = \la \Inc \cdot x_{\mu_1}^{a_1}x_{\mu_2}^{a_2}\cdots x_{\mu_r}^{a_r} \ra$. For example, if $I = \la \Inc \cdot x_3^2 x_5^4 x_8\ra$, then one gets  
\[
I_n = \begin{cases}
\la x_{i_1}^2 x_{i_2}^4 x_{i_3} \; | \; 3 \le i_1, i_2 - i_1 \ge 2, i_3 - i_2 \ge 3,  \ i_3 \le n \ra & \text{ if } n \ge 8 \\
0 & \text{ if } 0 \le n < 8. 
\end{cases}
\]
As a special case of our main result (see Theorem \ref{thm2}), one gets for such ideals:  

\begin{theorem}
      \label{thm:intro}
If $I = \la \Inc \cdot x_{\mu_1}^{a_1}x_{\mu_2}^{a_2}\cdots x_{\mu_r}^{a_r} \ra$, then       
\[ 
H_{K[X]/I} (s,t) = \displaystyle\frac{g (s,t)}{(1-t)^{\mu_r-1} \prod\limits_{j=1}^{r}\Big[ 1-s \cdot (1+t+...+t^{a_j-1})\Big]}, 
\] 
where $g (s, t) \in \ZZ[s, t]$ is a polynomial that is not divisible by any of the indicated irreducible factors of the denominator. 
\end{theorem}

We also determine the numerator polynomial $g(s, t)$ (see Theorem \ref{thm1}). For instance, if  $I = \la \Inc \cdot x_3^2 x_5^4\ra$ one gets (see, e.g.,  Example \ref{exa:numerator pol small r basic case})
\begin{align}
   \label{eq:expl Hilb series}
   H_{K[X]/I} (s, t) &= \frac{ (1-t)^4 + s (1-t)^3 (-1 + t^2 + t^4) + s^2 t^6 (1-t)^2 + s^3 t^6  (1-t) + s^4 t^6}{(1-t)^4 \cdot  [1 - s (1+t)] \cdot [1 - s (1+t + t^2 +t^3)]}. 
\end{align}

The Hilbert series in the case of  an arbitrary mononomial when $c \ge 1$ is qualitatively of the same form as in the case where $c=1$ (see Theorem \ref{thm2}). 

Let us compare the above result with the form of the equivariant Hilbert series of an arbitrary $\Inc$-invariant ideal as given in Equation \eqref{eq: general equiv Hilb series}. Example 7.3 in \cite{NR}  shows that there is no a priori bound on the degree of the polynomials $f_j$ appearing in the denominator and that they can have negative coefficients. Theorem \ref{thm:intro} establishes that the number of irreducible factors in the denominator can be arbitrarily large. Thus, the description of the denominator in Equation \eqref{eq: general equiv Hilb series} seems rather efficient. 

It is instructive to compare our results with the case of a noetherian graded hypersurface ring $A = K[y_1,\ldots,y_m]/\la f\ra$. It is a Cohen-Macaulay ring of dimension $m-1$, and  its multiplicity (degree) is $\deg f$. This information can be read off from its Hilbert series, which is $H_A (t) = \frac{1+t+ \cdots + t^{\deg f -1}}{(1-t)^{m-1}}$.   If $I$ is generated by the $\Inc$-orbit of a monomial, then $\dim K[X_n]/I_n = n (c-1) + \mu_r-1$, and so the growth is dominated by $c-1$. However, the degrees of the ideals $I_n$ eventually grow exponentially in $n$, and if $c=1$ (so $I = \la \Inc \cdot x_{\mu_1}^{a_1}x_{\mu_2}^{a_2}\cdots x_{\mu_r}^{a_r} \ra$)  the growth rate is dominated by 
\[
\lim_{n \to \infty} \sqrt[n]{\deg I_n} = \max  \{a_1,...,a_r\},  
\]
which is \emph{not} the degree of the orbit generator if $r \ge 2$. Again, there is a similar formula in the general case $c \ge 1$ (see Corollary \ref{cor:asympt degree}). Notice that even though each  $K[X_n]/I_n$ is Cohen-Macaulay the numerator polynomial of the Hilbert series of $K[X]/I$ in reduced form can have negative coefficients, as it is the case in Formula \eqref{eq:expl Hilb series}. However, the polynomials $f_j$ appearing in the irreducible factors of the denominator have only non-negative coefficients (see also Remark \ref{rem:non-neg coeff}). 

The proofs of rationality of an equivariant Hilbert series in \cite{NR} and \cite{KLS} both lead to an algorithm for computing it. However, here we develop a different method that makes the computations efficient. This is first carried out in Section \ref{sec:c=1} if $c =1$. We discuss this simpler case separately in order to stress the ideas and to simplify notation. 
The general case is treated in Section \ref{sec:gen case}. In some sense we are able to reduce it to the case where $c = 1$.

%%%%%%%%%%%%%%%%%%%%%%%%%%%%%%%

\section{A Special Case} 
\label{sec:c=1}

In this section we consider the special case where $c=1$, that is, the ring $K[X]$ has only one row of variables. Thus, we simplify notation and let $K[X] = K[x_j \: | \; j \in \NN]$. Any monomial in $K[X]$ can be written as $x_{\mu_1}^{a_1}x_{\mu_2}^{a_2}\cdots x_{\mu_r}^{a_r}$, where $\mu_1 < ... < \mu_r$ and $r, a_1,..,a_r \in \NN$. The $\Inc$-invariant  ideal $I$ of $K[X]$ generated by the orbit of this monomial is
\[
I = \la \Inc \cdot x_{\mu_1}^{a_1}x_{\mu_2}^{a_2}\cdots x_{\mu_r}^{a_r} \ra.
\]
Set $\underline{\mu} = (\mu_1,\ldots,\mu_r)$. 

Denote the set of non-negative integers by $\NN_0$. So, for $n \in \NN_0$, one has $K[X_n] = K[x_j \; | \; 1 \le j \le n]$. In particular, $K[X_0] = K$. Since $\Inc$ acts on $K[X]$ by $\pi \cdot x_j = x_{\pi(j)}$, we get the following explicit description of the ideal $I_n = I \cap K[X_n]$: 
\[
I_n = \begin{cases} 
\langle x_{i_1}^{a_1}x_{i_2}^{a_2}\cdots x_{i_r}^{a_r}  \hspace*{.1in} | \hspace*{.1in} \mu_1 \le i_1,\ i_r \leq n,\  \text{ and }  i_{j+1}-i_{j} \ge \mu_{j+1} - \mu_j \text{ for each } j  \rangle  &\text{ if } n <  \mu_r \\  
0, &\text{ if }0 \leq \mu_r <n,
\end{cases} 
\] 
Similarly, if $r\geq 2$, we also consider the ideal 
 \[
 J= \langle \text{Inc} \cdot x_{\mu_1}^{a_1}x_{\mu_2}^{a_2}\cdots x_{\mu_{r-1}}^{a_{r-1}}\rangle \subset K[X]
 \]
and $J_n = J\cap K[X_n]$  for $n \in \NN_0$.  The above description of the ideals $I_n$ immediately  gives the following simple, but very useful observation.

%%%%%%%%%%%%%%%%%%%%%%%%%%%%%%%%
%%%%%%         Lemma 1
%%%%%%%%%%%%%%%%%%%%%%%%%%%%%%%%
\begin{lemma}
     \label{lem1}
If $n\geq 1$, then 
\[ 
I_n = \langle I_{n-1} \rangle_{K[X_n]} + x_n^{a_r} \langle J_{n-\delta_r}\rangle _{K[X_n]},
\]
where $\delta_r := \mu_r-\mu_{r-1}\ge 1$ and $J_n$ is defined as the zero ideal if $n < 0$. 
\end{lemma}

Recall that the \emph{Hilbert series} of a proper homogeneous ideal $\mfa$ of $K[X_n]$ is defined as the formal power series 
\[
H_{K[X_n]/\mfa} (t) = \sum_{j \ge 0} \dim_K [K[X_n]/\mfa]_j \cdot t^j. 
\]
Hilbert showed that it is a rational function of the form $H_{K[X_n]/\mfa} (t) = \frac{f(t)}{(1-t)^d}$, where $f(t) \in \ZZ[t]$ and $d \in \NN_0$. We say that $H_{K[X_n]/\mfa}(t)$ is in \emph{reduced} form if the numerator and denominator are relatively prime or, equivalently, if $f(1) \neq 0$. In this case  $d$ is the Krull dimension of $K[X_n]/\mfa$ and $f(1) \ge 1$ is the \emph{degree} of $\mfa$ or \emph{multiplicity} of $K[X_n]/\mfa$.  In particular, the zero ideal has degree one. 

%%%%%%%%%%%%%%%%%%%%%%%%%%%%%%%%
%%%%%            Corollary 2
%%%%%%%%%%%%%%%%%%%%%%%%%%%%%%%%
\begin{corollary}
       \label{cor2}
\begin{itemize}
\item[(a)] If $n\geq \mu_r$, then $A_n := K[X_n]/I_n$ is a Cohen-Macaulay ring of dimension $\mu_r-1$. 

\item[(b)]       Setting $B_n := K[X_n]/J_n$, one gets for the Hilbert series if $n\ge \delta_r$
\[
H_{A_n}(t)  = (1+t+\cdots+t^{a_r-1})H_{A_{n-1}}(t) + \displaystyle\frac{t^{a_r}}{(1-t)^{\delta_r}}H_{B_{n-\delta_r}}(t). 
\]
\end{itemize}       
 \end{corollary}

\begin{proof}
Consider multiplication by $x_n^{a_r}$ on $A_n$. Lemma \ref{lem1} shows that, for $n \ge 1$, it induces a short exact sequence 
\begin{align} 
       \label{sesGen}
0 \to (K[X_n]\Big/ \langle J_{n-\delta_r}\rangle_{K[X_n]}) (-a_r) \to A_n \to  K[X_n]\Big/ \langle I_{n-1}, x_n^{a_r}\rangle_{K[X_n]} \to 0
\end{align} 
Since  the generators of the ideal $J_{n-\delta_r}$ are in $K[X_{n-\delta_r}]$, we get
\[
K[X_n]/ \langle J_{n-\delta_r}\rangle_{K[X_n]} \cong \begin{cases} 
K[X_n] & \text{ if }  0\leq n < \delta_r \\
 B_{n-\delta_r}[x_{n-\delta_r+1},...,x_n]  &\text{ if } n\geq \delta_r 
 \end{cases}
 \]
Observe also that $K[X_n]\Big/ \langle I_{n-1}, x_n^{a_r}\rangle_{K[X_n]} \cong A_{n-1} \otimes_K K[x_n]/(x_n^{a_r})$, which implies
\begin{align*}
H_{K[X_n]/ \langle I_{n-1}, x_n^{a_r}\rangle_{K[X_n]}} (t) & = H_{A_{n-1}} (t) \cdot H_{K[x_n]/(x_n^{a_r})}(t)  \\
& =  H_{A_{n -1}} (t) \cdot ( 1+t+ \cdots + t^{a_r-1}). 
\end{align*} 
Now, Sequence \eqref{sesGen} gives Claim (b). 

For proving (a), we use induction on $r\geq 1$. Let $r=1$. If $n \ge \mu_1$, then note that  $A_n = K[X_n]/\la x_{\mu_1}^{a_1}, x_{\mu_1+1}^{a_1},\cdots, x_n^{a_1}\ra$, which has dimension $\mu_1-1$.  If $r\geq 2$ and $n \ge \mu_r$, then the induction hypothesis   gives
\[
\dim K[X_n]\Big/ \langle J_{n-\delta_r}\rangle_{K[X_n]} = \dim B_{n-\delta_r} + \delta_r = \mu_{r-1} -1 + \delta_r = \mu_r -1. 
\]
The above Hilbert series computation also yields $\dim K[X_n]\Big/ \langle I_{n-1}, x_n^{a_r}\rangle_{K[X_n]}  = \dim A_{n-1}$. 
Thus, Claim (a) follows from Sequence \eqref{sesGen}.
\end{proof}

%%%%%%%%%%%%%%%%%%%%%%%%%%%%%%%%
%%%%%%           Remark 3 
%%%%%%%%%%%%%%%%%%%%%%%%%%%%%%%%
\begin{remark} 
         \label{rm3} 
In terms of Gorenstein liaison theory, Lemma \ref{lem1} says that $I_n$ is a basic double link of $\langle J_{n-\delta_r} \rangle_{K[X_n]}$ on  
$\langle  I_{n-1}\rangle_{K[X_n]}$. The  name stems from the fact that $I_n$ can be obtained from $\langle J_{n-\delta_r}\rangle_{K[X_n]}$ by two  Gorenstein links if $K[X_n]/\langle J_{n-\delta_r}\rangle _{K[X_n]}$ is generically Gorenstein (see \cite[Proposition 5.10]{KMMNP}).
\end{remark}

We are ready to establish the main result of this section. 

%%%%%%%%%%%%%%%%%%%%%%%%%%%%%%%%
%%%%%%            Theorem 1
%%%%%%%%%%%%%%%%%%%%%%%%%%%%%%%%
\begin{theorem}
       \label{thm1}
The equivariant Hilbert series of $A = K[X]/I$ is 
\[ 
H_A(s,t) = \displaystyle\frac{g_{r, \underline{a},\underline{\mu}}(s,t)}{(1-t)^{\mu_r-1} \prod\limits_{i=1}^{r}\Big[ 1-s(1+t+...+t^{a_i-1})\Big]}, 
\] 
where $g_{r, \underline{a}, \underline{\mu}} (s,t) \in \ZZ[s,t]$ is the polynomial with 
\[
g_{r, \underline{a}, \underline{\mu}} (s,t) \cdot (1-t-s)  = (1-t)^{\mu_r-r}\prod\limits_{i=1}^{r}(1-t-s+st^{a_i})-s^{\mu_r}t^{\sum\limits_{i=1}^{r}a_i}.
\]
Moreover, the above right-hand side is in reduced form, that is, the given numerator and denominator are relatively prime. 
\end{theorem}

\begin{proof}
Denote the right-hand side in the definition of $g_{r, \underline{a}, \underline{\mu}} (s,t)$ by $\widetilde{g}_{r, \underline{a}, \underline{\mu}} (s,t)$, that is, 
\[
\widetilde{g}_{r, \underline{a}, \underline{\mu}} (s,t) = (1-t)^{\mu_r-r}\prod\limits_{i=1}^{r}(1-t-s+st^{a_i})-s^{\mu_r}t^{\sum\limits_{i=1}^{r}a_i}.
\]
We first show by induction on $r \ge 1$
\begin{equation} 
     \label{eq:non-red HilbSer c=1}
H_A(s,t) = \displaystyle\frac{\widetilde{g}_{r, \underline{a},\underline{\mu}}(s,t)}{(1-t)^{\mu_r-1} (1-t-s) \prod\limits_{i=1}^{r}\Big[ 1-s(1+t+...+t^{a_i-1})\Big]}.
\end{equation}

Let $r=1$. One has $A_n = K[X_n]$ if $n < \mu_1$. If $n \ge \mu_1$, then we get
\[
A_n = K[X_n]/(x_{\mu_1}^{a_1}, x_{\mu_1+1}^{a_1},...,x_{n}^{a_1}) \cong K[X_{\mu_1 - 1}] \otimes_K \big (K[z]/\la z^{a_1} \ra \big)^{\otimes (n-\mu_1 +1)}. 
\]
Thus we obtain for the equivariant Hilbert series  
\begin{align*} 
  H_A(s,t) &= 
  \sum\limits_{n=0}^{\mu_1-1} \frac{1}{(1-t)^n}s^n + 
 \sum\limits_{n\geq \mu_1}  \frac{1}{(1-t)^{\mu_1-1}} (1+t+...+t^{a_1-1})^{n-\mu_1+1}\cdot s^{n}     &\\
 &= \sum\limits_{n=0}^{\mu_1-2}\Big( \frac{s}{1-t}\Big)^n + \Big( \frac{s}{1-t}\Big)^{\mu_1-1} \sum\limits_{n\geq \mu_1-	1}\Big[ s(1+t+...+t^{a_1-1})\Big]^{n-\mu_1+1} &\\
 &= \displaystyle\frac{1-\Big(\frac{s}{1-t}\Big)^{\mu_1-1}}{1-\frac{s}{1-t}} + \Big(\frac{s}{1-t}\Big)^{\mu_1-1}\displaystyle\frac{1}{1-s(1+t+...+t^{a_1-1})} &\\[1ex]
 &=   \displaystyle\frac{\big[(1-t)^{\mu_1-1}-s^{\mu_1-1}\big] \cdot \big[ 1-t -s(1-t^{a_1})\big] + s^{\mu_1-1}\big[1-t-s\big]}{(1-t)^{\mu_1-1}(1-t-s)\Big[1-s(1+t+....+t^{a_1-1})\Big]} &\\
 &= \displaystyle\frac{(1-t)^{\mu_1-1}(1-t-s+st^{a_1})-s^{\mu_1}t^{a_1}}{(1-t)^{\mu_1-1}(1-t-s)\Big[1-s(1+t+....+t^{a_1-1})\Big]}, 
 \end{align*}    
as desired. 

Let $r \ge 2$. Using Corollary \ref{cor2}(b), we get  
\begin{align*} 
 H_A(s,t)-1 &= \sum\limits_{n\geq 1} H_{A_n}(t)s^n \\  
                  &=  \sum\limits_{n=1}^{\delta_r -1} t^{a_r} \cdot H_{K[X_n]}(t)\cdot s^n  + \sum\limits_{n\geq \delta_r} \frac{t^{a_r}}{(1-t)^{\delta_r}}H_{B_{n-\delta_r}}(t)\cdot s^n \\
                  & \hspace*{.5cm} + \sum\limits_{n\geq 1} [1+t+...+t^{a_r-1}]\cdot H_{A_{n-1}}(t)\cdot s^n \\ 
 &= t^{a_r} \cdot \frac{s}{1-t} \cdot \displaystyle\frac{1-\Big(\frac{s}{1-t}\Big)^{\delta_r-1}}{1-\frac{s}{1-t}} + \frac{t^{a_r}}{(1-t)^{\delta_r}}s^{\delta_r}H_B(s,t) +  [1+t+...+t^{a_r-1}]\cdot s\cdot H_A(s,t). 
 \end{align*}
 Solving for the equivariant Hilbert series of $A$, a straight-forward computation gives the following recursive formula: 
\begin{eqnarray*}
     \label{recursion1}
H_A(s,t) & = 
&  \displaystyle\frac{1+ \frac{t^{a_r} s  \big [(1-t)^{\delta_r - 1} - s^{\delta_r -1} \big ]}{(1-t)^{\delta_r-1} (1-s-t)} + \frac{t^{a_r} s^{\delta_r}}{(1-t)^{\delta_r}}H_B(s,t) }
 {1-s[1+t+...+ t^{a_r-1}]}  %\\[1ex]
% &= 
% & \displaystyle\frac{(1-t)^{\delta_r}(1-s-t)+ st^{a_r}\big[ (1-t)^{\delta_r}-s^{\delta_r-1}(1-t) + s^{\delta_r-1}(1-s-t)\cdot H_B(s,t)\big]}
% {(1-t)^{\delta_r}(1-s-t)\big[ 1-s(1+t+...+t^{a_r-1})\big]} %\nonumber
 \end{eqnarray*}  
Applying the induction hypothesis to $B$ and noting $\mu_r = \mu_{r-1} + \delta_r$, we get
\begin{align*}
 H_A(s,t) \cdot [1-s\cdot(1+t+...+t^{a_r-1})]    
& = 1+\displaystyle\frac{t^{a_r} s (1-t)^{\mu_{r-1} } \big [(1-t)^{\delta_r - 1} - s^{\delta_r -1} \big ]}{(1-t-s)(1-t)^{\mu_r-1}} \\
& \hspace*{.5cm}+ \frac{t^{a_r}s^{\delta_r}(1-t)^{\mu_{r-1}-r+1}\prod\limits_{i=1}^{r-1}[1-t - s + s t^{a_i} ]-s^{\mu_r}t^{\sum\limits_{i=1}^{r}a_i}}{(1-t)^{\mu_{r}-1}(1-t-s)\prod\limits_{i=1}^{r-1}\big[1-s(1+t+...+t^{a_i-1}) \big]}. 
\end{align*}
Using $(1-t) \cdot [1-s(1+t+...+t^{a_i-1}) )] = [1-t - s + s t^{a_i} ]$, this gives 
\begin{align*}
\hspace{2em}&\hspace{-2em} H_A(s,t) \cdot (1-t)^{\mu_r-1} (1-s-t) \prod_{i=1}^{r} [1-s\cdot(1+t+...+t^{a_i-1})]    \\
& = H_A(s,t) \cdot (1-t)^{\mu_r-r} (1-s-t) [1-s\cdot(1+t+...+t^{a_r-1})]  \prod_{i=1}^{r-1} [1-t-s+s t^{a_i} ]\\
& =  - s^{\mu_r} t^{\sum_{i=1}^r a_i} +  \prod_{i=1}^{r-1} [1-t-s+s t^{a_i} ] \cdot \\
& \hspace*{.7cm}\left \{  (1-t)^{\mu_r - r} (1-s-t) + t^{a_r} s (1-t)^{\mu_{r-1} - r+1} \big [(1-t)^{\delta_r - 1} - s^{\delta_r -1} \big ]  + t^{a_r} s^{\delta_r} (1-t)^{\mu_{r-1} - r + 1}  \right \} \\
& =  - s^{\mu_r} t^{\sum_{i=1}^r a_i} +  \prod_{i=1}^{r-1} [1-t-s-+s t^{a_i} ] \cdot  (1-t)^{\mu_r-r} \left \{ 1 - t - s + s t^{a_r} \right \}.  
\end{align*}
Now Equation \eqref{eq:non-red HilbSer c=1} follows.  

It remains to show that $\widetilde{g}_{r, \underline{a}, \underline{\mu}} (s,t)$ is divisible by $(1-t-s)$ in $\ZZ[s,t]$, but not by any of the polynomials $[1 - s (1+t+\cdots+t^{a_i  - 1})]$. The first claim follows because 
\begin{align*} 
   \widetilde{g}_{r, \underline{a}, \underline{\mu}} (1-t, t)  & = (1-t)^{\mu_r-r} \prod\limits_{i=1}^{r}[(1-t)-(1-t) + (1-t) t^{a_i}] - (1-t)^{\mu_r} t^{\sum\limits_{j=1}^{r}a_j}  \\
&= (1-t)^{\mu_r-r} (1-t)^{r}t^{\sum\limits_{i=1}^{r}a_i} - (1-t)^{\mu_r}t^{\sum\limits_{j=1}^{r}a_j} = 0.  
\end{align*}
In order to show the other claims we compute 
\begin{align*}
\widetilde{g}_{r, \underline{a}, \underline{\mu}} \left (\frac{1}{1+t+\cdots+t^{a_i-1}} , t \right)   & = \widetilde{g}_{r, \underline{a}, \underline{\mu}} \left (\frac{1-t}{1 - t^{a_i}}, t \right ) =  - \left (\frac{1-t}{1 - t^{a_i}} \right )^{\mu_r} \cdot t^{\sum_{i=1}^r a_i}. 
\end{align*} 
Since this is not the zero polynomial, it follows that $[1 - s(1+t+\cdots+t^{a_i})]$ does not divide $\widetilde{g}_{r, \underline{a}, \underline{\mu}} (s,t)$, as desired. 
\end{proof}  

We give the numerator polynomial in the reduced form of the Hilbert series for small $r$. 

%%%%%%%%%
% Example to for a basic case.
%%%%%%%%%
\begin{example} 
   \label{exa:numerator pol small r basic case} 
For  $r = 1, 2, 3$, one gets
\begin{align*}
g_{1, \underline{a}, \underline{\mu}} (s,t) & = (1-t)^{\mu_1 - 1} + t^{a_1} \sum_{j=0}^{\mu_1-2} (1-t)^j s^{\mu_1 - 1- j}\\
g_{2, \underline{a}, \underline{\mu}} (s,t) & = (1-t)^{\mu_2-1} 
+ s (1-t)^{\mu_2 - 2} (-1 + t^{a_1} + t^{a_2}) 
+ t^{a_1+a_2} \cdot \sum\limits_{j=0}^{\mu_2-3} (1-t)^{ j} s^{\mu_2-1-j} \\
g_{3, \underline{a}, \underline{\mu}} (s,t) & = (1-t)^{\mu_3-1} 
+ s (1-t)^{\mu_3 - 2} (-2 + t^{a_1} + t^{a_2} + t^{a_3}) \\
& \hspace*{-.5cm} + s^2 (1-t)^{\mu_3 - 3} (1 - t^{a_1} - t^{a_2} - t^{a_3}+ t^{a_1 + a_2} + t^{a_1 + a_3} + t^{a_2 + a_3})    
 + t^{a_1+a_2+a_3} \cdot \sum\limits_{j=0}^{\mu_3-4} (1-t)^{ j} s^{\mu_3-1-j}. 
\end{align*} 
Here we use the convention that a sum $\sum_{j=0}^e$ is defined to be zero if $e < 0$. 
\end{example}  

We can also use our methods to determine the degree of each ideal $I_n$. 

%%%%%%%%%%%%%%%%%%%%%%%%%%%%%%%%
%%%%%%%%%%%        computing degree I_n
%%%%%%%%%%%%%%%%%%%%%%%%%%%%%%%%
\begin{proposition} 
            \label{prop:degree I_n}
If $n\ge \mu_r -1$,  then $\deg I_n$ is the coefficient of $t^{n-\mu_r+1}$ in the power series $\prod_{i=1}^{r} \frac{1}{1-a_it}$. In other words,
\[\prod_{i=1}^{r} \frac{1}{1-a_it} = \sum\limits_{n\geq \mu_r-1} \deg I_n\cdot t^{n-\mu_r+1}\]
\end{proposition}

\begin{proof} 
One can deduce this from Theorem \ref{thm1}. However, there is an easier,  more direct approach. 

Since $I_{\mu_r - 1} = 0$ by definition, we get $\deg I_{\mu_r - 1} = 1$ for each $r \ge 1$, as claimed. To determine $\deg I_n$ for larger $n$, we use induction on $r\geq 1$. If $r=1$, then $ I_n = \langle x_{\mu_1}^{a_1}, x_{\mu_r+1}^{a_1}, ..., x_n^{a_1}\rangle$, and so $\deg I_n = a_1^{n-\mu_1+1}$. Now the geometric series gives the claim, that is, $\sum\limits_{n\geq \mu_1-1} \deg I_n t^{n-\mu_1+1} = \sum\limits_{n\geq 0}  a_1^nt^n = \frac{1}{1-a_1t}$.

Let $r\geq 2$. If $n \ge \delta_r$, then Lemma \ref{lem1} gives 
\begin{equation} 
   \label{eq:deg recursion}
\deg I_n = a_r \deg I_{n-1} + \deg J_{n-\delta_r}.
\end{equation}
By induction on $r$, one has 
\[
\prod\limits_{i=1}^{r-1} \frac{1}{1-a_it} = \sum\limits_{n\geq \mu_{r-1}-1}\deg J_n t^{n-\mu_{r-1}+1} 
=\sum\limits_{n-\delta_r\geq \mu_{r-1}-1}\deg J_{n-\delta_r} t^{n-\mu_{r}+1}. 
\]
Hence we obtain,
\begin{align*}
\prod\limits_{i=1}^{r}\displaystyle\frac{1}{1-a_it} 
 &= \Big(\sum\limits_{n-\delta_r\geq \mu_{r-1}-1}\deg J_{n-\delta_r}\cdot t^{n-\mu_r+1}\Big) \cdot \Big( \sum\limits_{k\geq 0}a_r^kt^k\Big) \\
%&= \deg J_{\mu_{r-1}-1} + t\big( \deg J_{\mu_{r-1}} + a_r\deg J_{\mu_{r-1}-1}\big) + t^2\big( \deg J_{\mu_{r-1}} + a_r\deg J_{\mu_{r-1}-1}\big) + ...\\
&= \sum\limits_{n\geq \mu_r-1}\left [ \sum\limits_{i=0}^{n-\mu_r+1}a_r^{n-\mu_r+1-i}\cdot\deg J_{\mu_{r-1}-1+i}\right]t^{n-\mu_r+1}
\end{align*}
This implies our assertion because 
\[
\sum\limits_{i=0}^{n-\mu_r+1}a_r^{n-\mu_r+1-i}\cdot\deg J_{\mu_{r-1}-1+i} = \deg I_n. 
\]
Indeed, if $n= \mu_r-1$ then this formula is true since 
$a_r^{0}\cdot\deg J_{\mu_{r-1}-1} = 1 = \deg I_{\mu_r-1}$. Let $n\geq \mu_r$. Using Equation \eqref{eq:deg recursion}, one has 
\begin{align*}
\sum\limits_{i=0}^{n-\mu_r+1}a_r^{n-\mu_r+1-i}\cdot\deg J_{\mu_{r-1}-1+i} &= \deg J_{n-\mu_r+\mu_{r-1}} + a_r \sum\limits_{i=0}^{n-\mu_r} a_r^{n-\mu_r-i} \deg J_{\mu_{r-1}-1+i} \\
&= \deg J_{n-\delta_r} + a_r  \deg I_{n-1} = \deg I_n, 
\end{align*}
as desired. 
\end{proof}

One can use the last result to explicitly compute $\deg I_n$. This is easiest if $a_1,...,a_r$ are pairwise distinct.  

\begin{corollary} 
    \label{Remark-computationDegree} 
If $a_1,...,a_r$ are pairwise distinct, then \; $\deg I_n = {\displaystyle \sum\limits_{i=1}^r \frac{a_i^{n-\mu_r+r}}{\prod\limits_{j\neq i} (a_i-a_j)}}$, provided $n \ge \mu_r - 1$.   
\end{corollary} 
      
\begin{proof}
Using partial fractions, one can write
\[ \prod\limits_{i=1}^{r} \frac{1}{1-a_it} = \displaystyle\frac{C_1}{1-a_1t} + ...+ \displaystyle\frac{C_r}{1-a_rt},  
\]
where
%\[ 
%C_i = \mathlarger{\prod\limits_{j\neq i}} \displaystyle\frac{1}{1-\frac{a_j}{a_i}} = \displaystyle\frac{a_i^{r-1}}{\prod\limits_{j\neq i}(a_i-a_j)}.  
%\]
\[ 
C_i = \prod\limits_{j\neq i} \displaystyle\frac{1}{1-\frac{a_j}{a_i}} = \displaystyle\frac{a_i^{r-1}}{\prod\limits_{j\neq i}(a_i-a_j)}.  
\]
Hence
%\begin{align*}
%\mathlarger{\prod\limits_{i=1}^{r}}\displaystyle\frac{1}{1-a_it} &= \sum\limits_{i=1}^{r}\displaystyle\frac{a_i^{r-1}}{\prod\limits_{j\neq i}(a_i-a_j)}\cdot\displaystyle\frac{1}{1-a_it} = \sum\limits_{i=1}^{r} \Big[ \displaystyle\frac{a_i^{r-1}}{\prod\limits_{j\neq i}(a_i-a_j)}\cdot\sum\limits_{k\geq 0}a_i^kt^k\Big]
%\end{align*}
\begin{align*}
\prod\limits_{i=1}^{r}\displaystyle\frac{1}{1-a_it} &= \sum\limits_{i=1}^{r}\displaystyle\frac{a_i^{r-1}}{\prod\limits_{j\neq i}(a_i-a_j)}\cdot\displaystyle\frac{1}{1-a_it} = \sum\limits_{i=1}^{r} \Big[ \displaystyle\frac{a_i^{r-1}}{\prod\limits_{j\neq i}(a_i-a_j)}\cdot\sum\limits_{k\geq 0}a_i^kt^k\Big]
\end{align*}
Now we conclude by  Proposition  \ref{prop:degree I_n}. 
\end{proof}

 %%%%%%%%%%%%%%%%%%%%%%%%%%%%%%%%%%
 %%%%%%%%%%%%%%%%%%%%%%%%%%%%%%%%%%
 \section{The General Case}
 \label{sec:gen case}
 
 We extend the results of the previous section. We use the notation established in the introduction. So we fix an integer $c \ge 1$ and consider the polynomial rings $K[X_n] =  K[x_{i,j} \; | \; 1 \le i \le c,\ 1 \le j \le n]$ and $K[X] =  K[x_{i,j} \; | \; 1 \le i \le c,\ 1 \le j]$. Any monomial of positive degree in $K[X]$ can be written as 
 \[
 x^{\underline{a}} = \prod_{i = 1}^c \prod_{j=1}^s x_{i,j}^{a_{i,j}}, 
 \]
 where $\underline{a} = (a_{i,j})$ is a $c \times s$ non-zero matrix whose entries are nonnegative integers. Denote the indices of the non-zero columns of $\underline{a}$ by $\mu_1,\ldots,\mu_r$, where $\mu_1 < \mu_2 < ... < \mu_r$. We may assume that the last column of $\underline{a}$ is not zero, that is, $\mu_r = s$ and $\underline{a} \in \NN_0^{c \times \mu_r}$. Thus, we can rewrite $x^{\underline{a}}$ more explicitly as
\[
x^{\underline{a}} = 
(x_{1,\mu_1}^{a_{1,\mu_1}}\cdots x_{1,\mu_r}^{a_{1,\mu_r}}) \cdot 
(x_{2,\mu_1}^{a_{2,\mu_1}}\cdots x_{2,\mu_r}^{a_{2,\mu_r}}) \cdots (x_{c,\mu_1}^{a_{c,\mu_1}}\cdots x_{c,\mu_r}^{a_{c,\mu_r}}). 
\]
Put $\underline{\mu}  = (\mu_1,\ldots,\mu_r)$. 
 
In order to determine the equivariant Hilbert  series of $K[X]/I$, where $I = \la \Inc \cdot x^{\underline{a}} \ra$, we also consider the ideal 
\[
J= \langle \text{Inc} \cdot \prod\limits_{i=1}^{c} \prod\limits_{j=1}^{\mu_{r-1}}x_{i,j}^{a_{i,j}}\rangle
\]
 if $r \ge 2$. Thus, we get for $I_n = I \cap K[X]$ and $J_n = J \cap K[X]$ that 
 $I_n = 0$ if $n < \mu_r$ and that $J_n = 0$ if $n < \mu_{r-1}$. Moreover, there is again a useful  relation among these ideals. 
 
 \begin{lemma}
     \label{rec-In-ses}
If $n\geq 1$, then 
\[ 
I_n := \langle I_{n-1} \rangle_{K[X_n]} + \prod\limits_{i=1}^{c} x_{i,n}^{a_i,\mu_r}  \langle J_{n-\delta_r}\rangle _{K[X_n]},
\]
where $\delta_r := \mu_r-\mu_{r-1}\ge 1$. 
\end{lemma} 

It follows that $I_n$ is a basic double link of $J_{n-\delta_r}$ on $I_{n-1}$ because of the following consequence. We use the notation $A_n = K[X_n]/I_n$, \ $B_n = K[X_n]/J_n$, and $b_j = \sum\limits_{i=1}^{c}a_{i,\mu_j}$ for $j=1,\ldots,r$.  Thus, $b_j$ is the total degree of the divisor of $x^{\underline{a}}$ whose factors are the variables appearing in column $\mu_j$. 

\begin{corollary}
       \label{cor:comp In and Jn}
\begin{itemize}
\item[(a)] If $n\geq \mu_r$, then $A_n$ is a Cohen-Macaulay ring of dimension $n (c-1) + \mu_r-1$. 

\item[(b)]        If $n\ge \delta_r$, then one has for the Hilbert series
\[
H_{A_n}(t) =  \frac{1-t^{b_r}}{(1-t)^c} H_{A_{n-1}}(t) + \frac{t^{b_r}}{(1-t)^{c\delta_r}}H_{B_{n-\delta_r}}(t). 
\]
\end{itemize}       
 \end{corollary} 
 
 \begin{proof}
 Multiplication by $\prod\limits_{i=1}^{c} x_{i,n}^{a_i,\mu_r}$ on $A_n$ induces the exact sequence 
 \[ 
 0 \to K[X_n]\Big/ \langle J_{n-\delta_r}\rangle_{K[X_n]}(-b_r) \to A_n \to  K[X_n]\Big/ \langle I_{n-1}, \prod\limits_{i=1}^{c} x_{i,n}^{a_i,\mu_r}  \rangle_{K[X_n]} \to 0. 
 \]
 Furthermore, we have 
 \[ K[X_n]\Big/ \langle J_{n-\delta_r}\rangle_{K[X_n]} \cong \begin{cases} K[X_n]  &\text{ if } 0\leq n < \delta_r \\
 B_{n-\delta_r}\big[ x_{i,j} \, | \; 1 \le i \le c, n-\delta_r < j \le n]  &\text{ if  } n\geq \delta_r. 
\end{cases}
\]
Now the claims follow as in the proof of Corollary \ref{cor2}. 
 \end{proof} 
 
 Our main result is the promised extension of Theorem \ref{thm1}.

%%%%%%%                  %%%%%%%
%%%%%%% Theorem 2 %%%%%%%
%%%%%%%                  %%%%%%%
\begin{theorem} 
      \label{thm2} 
Setting $\underline{b} = (b_1,\ldots,b_r)$, 
 the  equivariant Hilbert series of $A = K[X]/I$ is 
\[
H_A(s,t) = \frac{g_{r, c,  \underline{b}, \underline{\mu}} (s,t)} { (1-t)^{c(\mu_r-r-1)+r} \prod\limits_{j=1}^{r}\big[ (1-t)^{c-1} - s(1+t+...+t^{b_j-1})\big]}, 
\]
where $g_{r, c,  \underline{b}, \underline{\mu}} (s,t) \in \ZZ[s, t]$ is the polynomial with 
\[
g_{r, c,  \underline{b}, \underline{\mu}}(s,t) \cdot \big[ (1-t)^c-s\big]  = (1-t)^{c(\mu_r-r)} \prod\limits_{j=1}^{r} \big[ (1-t)^c-s+st^{b_j}\big] - s^{\mu_r}t^{\sum\limits_{j=1}^{r} b_j}.
\]

Furthermore, the above rational function is in reduced form, that is, the given numerator and denominator polynomials are relatively prime. (Notice that the exponent $[c(\mu_r-r-1)+r]$ of $(1-t)$ is negative if and only if $r < c$ and $\mu_r = r$.) 
\end{theorem}

\begin{proof} We argue as in the proof of Theorem \ref{thm1}. Set 
\[
\widetilde{g}_{r, c, \underline{a}, \underline{\mu}} (s,t) = (1-t)^{c(\mu_r-r)} \prod\limits_{j=1}^{r} \big[ (1-t)^c-s+st^{b_j}\big] - s^{\mu_r}t^{\sum\limits_{j=1}^{r} b_j}.
\]
Using induction on $r \ge 1$, one shows 
\begin{equation} 
     \label{eq:non-red HilbSer any c}
H_A(s,t) = \displaystyle\frac{\widetilde{g}_{r, c, \underline{a},\underline{\mu}}(s,t)}{ (1-t)^{c(\mu_r-r-1)+r} \big[ (1-t)^c-s\big]\prod\limits_{j=1}^{r}\big[ (1-t)^{c-1} - s(1+t+...+t^{b_j-1})\big]}.
\end{equation}
Indeed, let $r = 1$. If $n \ge \mu_i$, then we get 
\[
A_n \cong K[X_{ \mu_1-1} ] \otimes \Big( K[z_1,..,z_c]/(z_1^{a_{1,\mu_1}} \cdots z_c^{a_{c,\mu_1}} \Big)^{ \otimes n-\mu_1+1}, 
\]
where $z_1,\ldots,z_c$ are new variables. It follows that 
\[
H_A(s,t) = \sum\limits_{n=0}^{\mu_1-1}\frac{1}{(1-t)^{nc}}s^n + \sum\limits_{n\geq \mu_1} \frac{1}{(1-t)^{c(\mu_1-1)}}\Big(\frac{1+t+ ... +t^{b_1-1}}{(1-t)^{c-1}} \Big)^{n-\mu_1+1} s^n.   
\]
Now a computation as in the proof of Theorem \ref{thm1} gives the desired formula. 

Let $r \ge 2$. Corollary \ref{cor:comp In and Jn} implies 
 \begin{align*} 
 \hspace{2em}&\hspace{-2em} 
 H_A(s,t)-1 = \sum\limits_{n\geq 1} H_{A_n}(t)s^n \\  
                  &=  \sum\limits_{n=1}^{\delta_r -1} t^{b_r} \cdot H_{K[X_n]}(t)\cdot s^n  
                  + \sum\limits_{n\geq \delta_r} \frac{t^{b_r}}{(1-t)^{c\delta_r}}H_{B_{n-\delta_r}}(t)\cdot s^n 
                  + \sum\limits_{n\geq 1} \frac{1+t+ ... +t^{b_r-1}}{(1-t)^{c-1}}  \cdot H_{A_{n-1}}(t)\cdot s^n 
                  \\ 
 &=   t^{b_r} \cdot \frac{s}{(1-t)^c} \cdot \displaystyle \frac{1-\Big(\frac{s}{(1-t)^c}\Big)^{\delta_r-1}}{1-\frac{s}{(1-t)^c}} 
 + \frac{t^{b_r}}{(1-t)^{c\delta_r}} s^{\delta_r} H_B(s,t) + \frac{s ( 1+t+ ... +t^{b_r-1}) }{(1-t)^{c-1}} \cdot H_A(s,t). 
 \end{align*}
 This gives 
 \begin{align*}
\hspace{2em}&\hspace{-6em} 
 H_A(s,t) \cdot\frac{(1-t)^{c-1} -s(1+t+...+t^{b_r-1})}{(1-t)^{c-1}}  \\
 & =  1+ t^{b_r}s\frac{(1-t)^{c(\delta_r-1)}-s^{\delta_r-1}}{(1-t)^{c(\delta_r-1)}[(1-t)^{c}-s]}+ \frac{t^{b_r}s^{\delta_r}}{(1-t)^{c\delta_r}}H_B(s,t) 
 \end{align*}
 Applying the induction hypothesis to $B$, a  computation similar to the one in the proof of Theorem~\ref{thm1} establishes Equation \eqref{eq:non-red HilbSer any c}. 
 
 It remains to show that $\widetilde{g}_{r, c, \underline{a}, \underline{\mu}} (s,t)$ is divisible by $((1-t)^c-s)$ in $\ZZ[s,t]$, but not by any of the polynomials $[(1 -t)^{c-1} - s (1+t+\cdots+t^{a_i  - 1})]$. The first claim is true  because 
 \begin{align*}
\widetilde{g}_{r, c, \underline{a}, \underline{\mu}} ((1-t)^c,t) 
& =  (1-t)^{c(\mu_r-r)} \prod\limits_{i=1}^{r}[(1-t)^c-(1-t)^c + (1-t)^ct^{b_i}] - (1-t)^{c\mu_r}t^{\sum\limits_{j=1}^{r}b_j}  \\
&= (1-t)^{c(\mu_r-r)} (1-t)^{rc}t^{\sum\limits_{i=1}^{r}b_i}] - (1-t)^{c\mu_r}t^{\sum\limits_{j=1}^{r}b_j} = 0. 
\end{align*}
Substituing ${\displaystyle s = \frac{(1-t)^{c-1}}{1+t+...+t^{b_r-1}} = \frac{(1-t)^c}{1-t^{b_r}}}$, we get 
\begin{align*}
%\hspace{2em}&\hspace{-2em} 
\widetilde{g}_{r, c, \underline{a}, \underline{\mu}}  \left (\frac{(1-t)^{c-1}}{1+t+...+t^{b_1-1}}, t \right) 
%=  (1-t)^{c(\mu_r-r)} \prod\limits_{i=1}^{r}[(1-t)^c-\frac{(1-t)^c}{1-t^{b_1}}(1-t^{b_i})] - \Big(\frac{(1-t)^c}{1-t^{b_1}}\Big)^{\mu_r}t^{\sum\limits_{j=1}^{r}b_j}  \\
%&= (1-t)^{c(\mu_r-r)}[(1-t)^c-\frac{(1-t)^c}{1-t^{b_1}}(1-t^{b_1})]  \prod\limits_{i=2}^{r}[(1-t)^c-\frac{(1-t)^c}{1-t^{b_1}}(1-t^{b_i})] - \Big(\frac{(1-t)^c}{1-t^{b_1}}\Big)^{\mu_r}t^{\sum\limits_{j=1}^{r}b_j}  \\
&
= - \frac{(1-t)^{(c-1) \mu_r}}{(1+t+...+t^{b_r-1})^{\mu_r}} \cdot t^{\sum\limits_{j=1}^{r}b_j}. 
\end{align*}
Since this is not the zero polynomial the argument is complete now. 
\end{proof}

Again we give the numerator polynomial in the reduced form of the Hilbert series for small $r$, where we assume that $c(\mu_r-r-1)+r \ge 0. $

%%%%%%%%%
% explicit enumerator polynomial 
%%%%%%%%%
\begin{example} 
   \label{exa:numerator pol small r}
For  $r = 1, 2, 3$, one has  
\begin{align*}
g_{1, c, \underline{a}, \underline{\mu}}  & =  (1-t)^{c (\mu_1 -1)} + t^{b_1} \cdot \sum_{j=0}^{\mu_1-2} (1-t)^{c j} s^{\mu_1 - 1- j}\\
g_{2, c, \underline{a}, \underline{\mu}} (s,t) & = (1-t)^{c (\mu_2-1)} 
+ s (1-t)^{c(\mu_2 - 2)} (-1 + t^{b_1} + t^{b_2}) 
+ t^{b_1+b_2} \cdot \sum\limits_{j=0}^{\mu_2-3} (1-t)^{c j} s^{\mu_2-1-j} \\
g_{3, c, \underline{a}, \underline{\mu}}  (s,t) & =  (1-t)^{c (\mu_3-1)} 
+ s (1-t)^{c(\mu_3 - 2)} (-2 + t^{b_1} + t^{b_2} + t^{b_3}) \\
& \hspace*{-.5cm} + s^2 (1-t)^{c (\mu_3 - 3)} (1 -t^{b_1} - t^{b_2} - t^{b_3}+ t^{b_1 + b_2} + t^{b_1 + b_3} + t^{b_2 + b_3})    
 + t^{b_1+b_2+b_3} \cdot \sum\limits_{j=0}^{\mu_3-4} (1-t)^{c j} s^{\mu_3-1-j}. 
\end{align*} 
Notice that these polynomials simplify if the $\mu_i$'s are as small as possible, that is, $\mu_i = i$. For example, then one gets  $g_{1, c, \underline{a}, \underline{\mu}} = 1$ and $g_{2, c, \underline{a}, \underline{\mu}} (s,t) = (1-t)^c 
+ s  (-1 + t^{b_1} + t^{b_2})$. 
\end{example}  

\begin{remark}
Observe the similarity of the formulas in Theorem \ref{thm1} and \ref{thm2}. Indeed, Theorem \ref{thm2} is formally obtained from Theorem \ref{thm1} by replacing each $a_j$ by the total column degree $b_j$ and $(1-t)$ by $(1-t)^c$. 
\end{remark}

Now we determine the degree of $I_n$. 

%%%%%%%%                                       %%%%%%%% 
%%%%%%%%  Degree of I_n for any c   %%%%%%%%
%%%%%%%%                                       %%%%%%%%
\begin{proposition} 
     \label{prop:degree general}
If $n\ge \mu_r -1$,  then $\deg I_n$ is the coefficient of $t^{n-\mu_r+1}$ in the power series $\prod_{j=1}^{r} \frac{1}{1-b_j t}$.    That is, 
\[
\prod\limits_{j=1}^{r}\frac{1}{1-b_j t} = \sum\limits_{n\geq \mu_r-1}\deg I_n\cdot t^{n-\mu_r+1}. 
\]
\end{proposition} 

\begin{proof} 
If $r\geq 2$ and $n \ge \delta_r$,  Lemma \ref{rec-In-ses} gives 
\begin{equation*} 
   \label{eq:deg recursion gen}
\deg I_n = b_r \deg I_{n-1} + \deg J_{n-\delta_r}.
\end{equation*}
Now we conclude as in the proof of Proposition \ref{prop:degree I_n}. 
\end{proof}      

Analogously to Corollary \ref{Remark-computationDegree} this gives the following explicit formula. 

\begin{corollary} 
    \label{Remark-computationDegree gen} 
If $b_1,...,b_r$ are pairwise distinct, then \; ${\displaystyle
\deg I_n = \sum\limits_{i=1}^r\displaystyle\frac{b_i^{n-\mu_r+r}}{\prod\limits_{j\neq i}(b_i-b_j)}  
}$,  provided $n \ge \mu_r - 1$.   
\end{corollary} 

For any $\Inc$-invariant ideal $I$ of $K[X]$, it is shown in \cite[Theorem 7.9]{NR} that the two limits $\lim\limits_{n \to \infty} \frac{\dim K[X_n]/I_n}{n}$ and  $\lim\limits_{n \to \infty} \sqrt[n]{\deg I_n}$
exist and are non-negative integers, where $I_n = I \cap K[X_n]$. Following  \cite[Remark 7.14]{NR},  we refer to these integers as the \emph{dimension}   of $K[X]/I$ and the \emph{degree} of $I$, respectively. If $I$ is generated by the orbit of a monomial, we obtain the following values.

\begin{corollary}
   \label{cor:asympt degree}
For $I = \la \Inc \cdot \prod_{i = 1}^c \prod_{j=1}^{\mu_r} x_{i,j}^{a_{i,j}} \ra$, one has 
\begin{itemize}   
\item[(a)] $\dim K[X]/I = c-1$. 

\item[(b)] $\deg I  = \max\{b_1,...,b_r\}. $
\end{itemize}
\end{corollary}

\begin{proof}
(a) is a consequence of Corollary \ref{cor:comp In and Jn}. 

(b) follows by using partial fractions as in \cite[Lemma A.3]{NR}. We leave the details to the interested reader.
\end{proof}

We conclude with some comments about non-negativity of the coefficients of the polynomials appearing in an equivariant Hilbert series. 

\begin{remark}
     \label{rem:non-neg coeff}
If $A$ is a graded Cohen-Macaulay quotient  of a noetherian polynomial ring, then it is well-known that the numerator polynomial in its reduced Hilbert series has non-negative coefficients only. 
We have seen above that in the case of an $\Inc$-invariant ideal $I$ of $K[X]$ the condition that all rings $K[X_n]/I_n$ are Cohen-Macaulay is not sufficient to guarantee that the numerator polynomial $g(s, t)$ in a reduced Hilbert series of $K[X]/I$ as in Equation \eqref{eq: general equiv Hilb series} has non-negative coefficients only (see, e.g., Example \ref{exa:numerator pol small r}).  However, the coefficients in the polynomials $f_j (t)$ appearing in the denominator of the Hilbert series all have non-negative coefficients if $I$ is generated by the orbit of a monomial. This suggests the following question.  
\end{remark} 

\begin{question}
Assume $I$ is an $\Inc$-invariant ideal of $K[X]$ such that each ring $K[X_n]/I_n$ is Cohen-Macaulay. Is it then true that the coefficients of the polynomials $f_j (t)$ appearing in the reduced form of the Hilbert series \eqref{eq: general equiv Hilb series} are all non-negative? 
\end{question}

%%%%%%%
%------------------------------------------
%
%
%--------------------------------------------------------
%%%%%%%%%%%%%%%%%%%%%%%%%%%%%%%
%%%%%%%%%%%%%%%%%%%%%%%%%%%%%%%
%%--------------------------------------------------------
%
%
%
%
%
%  REFERENCES
%--------------------------------------------------------

%\begin{thebibliography}{99}
%  
%\end{thebibliography}

\end{document}